\documentclass{article}
\usepackage{amsfonts}
\usepackage{dsfont}
\usepackage{amssymb}
\usepackage{newlfont}
\usepackage{amsthm}
\usepackage{euscript}
\usepackage{color}
\usepackage{amsmath}
\usepackage{graphicx}

\makeatletter
\renewcommand{\@makecaption}[2]{%
\vspace{\abovecaptionskip}%
\sbox{\@tempboxa}{#1. #2}
\ifdim\wd\@tempboxa >\hsize
   #1. #2\par
   \else
   \global\@minipagefalse
   \hbox to \hsize    {\hfil #1. #2\hfil}%
\fi
\vspace{\belowcaptionskip}}
\makeatother
\newtheorem{theorem}{Theorem}
\newtheorem{lemma}{Lemma}
\theoremstyle{remark}
\newtheorem{remark}{Remark}
\theoremstyle{definition}
\newtheorem{definition}{Definition}

\begin{document}
\large
\begin{center}
\bf \Large On Brownian motion on the plane with membranes on rays with a common endpoint
\end{center}
\vskip 10 pt
\begin{center}
\Large Olga V. Aryasova, 
Andrey Yu. Pilipenko\footnote{Research partially supported by Ministry of Education and
Science of Ukraine, Grant № F26/433-2008, and  National Academy of Sciences of Ukraine, Grant № 104-2008.}
\end{center}
\vskip 10 pt
\begin{abstract}
We consider a Brownian motion on the plane with semipermeable membranes on $n$ rays that have a common endpoint in the origin. We obtain the necessary and sufficient conditions for the process to reach the origin and we show that the probability of hitting the origin is equal to zero or one.
\end{abstract}

\section*{Introduction}

Let $c_1,\dots,\ c_n$ be given rays in $\mathbb{R}^2$ such that
they share a common endpoint. Suppose polar coordinates in
$\mathbb{R}^2$ are denoted by $(r,\varphi),r\geq0,\varphi\in[0,2\pi)$ and
$$
c_k=\{(r,\varphi):r\geq0,\ \varphi=\varphi_k\},
$$
where $0\leq\varphi_k<2\pi, \varphi_j\neq\varphi_k, j\neq k, j,k\in\{1,\dots,n\}$.

Let $n_k, k=1,\dots,n,$ denote the unit vector normal to $c_k$
that points  anticlockwise and let $v_k$ be a vector in $\mathbb
R^2$ such that $(v_k,n_k)=1$. The angle between $n_k$ and $v_k$
denoted by $\theta_k\in (-\frac{\pi}{2},\frac{\pi}{2})$ is
referred to as a positive if and only if $v_k$ points towards the
origin. The case of $n=5, \theta_1>0, \theta_2
>0$ is shown in Figure \ref{Fig. 1}.

\begin{figure}[h]
  \includegraphics[width=10 cm]{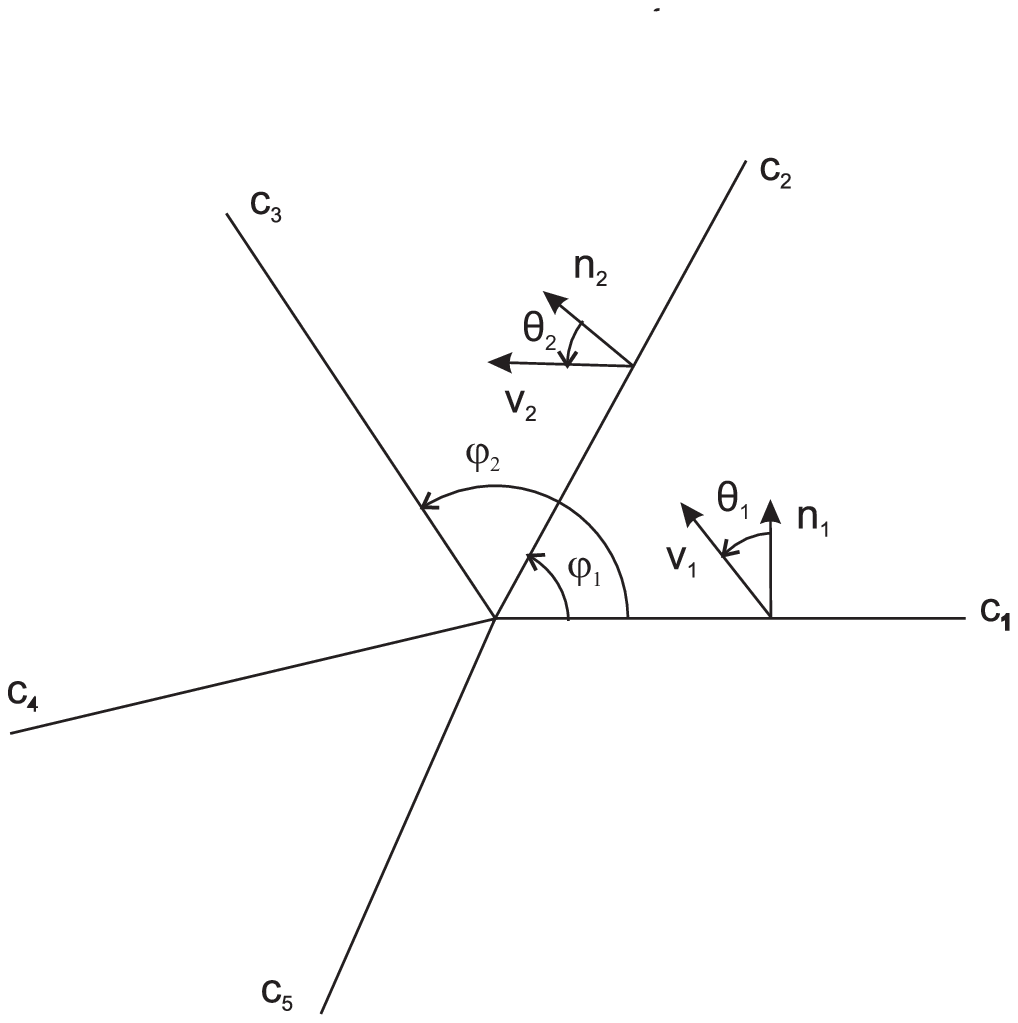}\\
  \caption{}
  \label{Fig. 1}
\end{figure}

Let $\gamma_k,\ k=1,\dots,n,$ be a real constant such that
$|\gamma_k|\leq1$.
Consider the stochastic differential equation in $\Bbb R^2$
\begin{equation}
dx(t)=dw(t)+\sum_{k=1}^n \gamma_k v_k dL_x^{c_k}(t) \label{main}
\end{equation}
with initial condition $x(0)=x^0,\ x^0\in\Bbb R^2.$

In this equation $(w(t))_{t\geq0}$ is a Wiener process in $\Bbb
R^2$, $(L_x^{c_k}(t))_{t\geq 0}$ is a local time of the process
$(x(t))_{t\geq0}$ at the ray $c_k$, and it is defined by the
formula
\begin{equation}
L_x^{c_k}(t)=\lim_{\varepsilon\downarrow0}\frac{1}{2\varepsilon}\int_0^t \label{L^c_k} {\hbox{1\!\!\!\!\;\,{I}}}_{A_{\varepsilon}^k}(x(s))ds,
\end{equation}
where $A_\varepsilon^k=\{x\in\Bbb R^2:\exists y\in c_k,s\in[-1,1] \mbox{ such that } x=y+\varepsilon s n_k$\}.

The solution of this equation can be regarded as a
Wiener process in $\mathbb R^2$ skewing on
the rays $c_1,\dots, c_n$ until the time it reaches the origin (see Section 1). The skew on $c_k$ is defined by $\gamma_k$ being  a coefficient of permeability. When $\gamma_k$ is equal to $1$ we have a
Wiener process reflecting instantaneously at $c_k$ in the
direction of $v_k$. If $\gamma_k$ is equal to $-1$ the process
reflects at $c_k$ in the direction of $-v_k$. If $\gamma_k\in (-1,1)$
a semipermeable membrane is on $c_k$.

The aim of the investigation is to show that the process hits the
origin, when starting away from it, with probability zero or one.
We obtain an explicit expression in variables
$\theta_1,\dots,\theta_n,\gamma_1,\dots,\gamma_n,\xi_1,\dots,\xi_n$
whose sign this probability depends on.

The case of the wedge (two rays) with reflection on its sides was studied by Varadhan and Williams \cite{VW}. They proved (in our notation) that the process does not reach the origin if $\theta_1-\theta_2\leq 0$ and does reach the origin if $\theta_1-\theta_2>0$. In this case our result coincides with that of Varadhan and Williams (see Example 2).

\section{Construction of the process}

In this Section we build the process which is a solution of
equation (\ref{main}) up to the time it hits the origin.

Conventionally, we put $\varphi_{nl+k}=\varphi_k, c_{nl+k}=c_k,$  $\gamma_{nl+k}=\gamma_k,$ $
k=1,\dots,n,\ l\in{\Bbb N}.$ Define $\xi_k=\varphi_{k+1}-\varphi_k$.  Let
$S_k=\{(r,\varphi):r>0, \ \varphi_k\leq\varphi\leq \varphi_{k+1}\}, k=1,\dots,n$,
be the wedge $c_k 0 c_{k+1}.$ Let $\varphi_1=0.$
\begin{remark}  Without loss of generality we can assume that
$\xi_k\leq\pi, k=1,\dots,n.$ Indeed, otherwise we can introduce
additional ray \label{Remark 1}
\begin{eqnarray*}
c_{ad}&=&\{(r,\varphi):r\geq0,\ \varphi=\pi\},
\end{eqnarray*}
and put $\gamma_{ad}=0.$
\end{remark}

So, throughout this Section we assume $\xi_k\leq\pi/2, k=1,\dots,n.$

\begin{lemma}
\label{Lemma 1}
Given $x^0\in \Bbb R^2$, there exists  a
unique strong solution of the equation (\ref{main}) up to the
first time of hitting the origin.
\end{lemma}
\begin{proof}
(i) The process  $(x(t))_{t\geq 0}$ is an ordinary Wiener
process up to the first time of hitting  one of the rays
$c_1,\dots,c_n$ . Define
$$
\tau_1=\inf\left\{t\geq 0: x(t)\in c_{j_1}\ \mbox{for some}\
j_1\in\{1,\dots,n\}\right\}.
$$
Without loss of generality,  let $x(\tau_1)\in c_1$. Until the
process reaches another ray only the local time at the ray $c_1$
can increase.  Put
$$
\tau_2=\inf\left\{t>\tau_1 : x(t)\in  c_{j_2}\ \mbox{for some}\
j_2\in\{1,\dots,n\},j_2\neq j_1\right\}.
$$
Then for all $t<\tau_2$ equation (\ref{main}) has the form
\begin{eqnarray}
dx_1(t)&=&dw_1(t)-\gamma_1 \tan\theta_1dL_{x}^{c_1}(t),\label{x_1}\\
dx_2(t)&=&dw_2(t)+\gamma_1dL_{x}^{c_1}(t),\label{x_2}
\end{eqnarray}
where $(x_1(t),x_2(t))$ are the cartesian coordinates of the
process $(x(t))_{t\geq 0}$, $(w_1(t))_{t\geq 0}$ and
$(w_2(t))_{t\geq 0}$ are independent one-dimensional Wiener
processes.  Note that by Remark \ref{Remark 1}
$$
L_{x}^{c_1}(t)=L_{x_2}^0(t):=\lim_{\varepsilon\downarrow
0}\frac{1}{2\varepsilon}\int_0^t{\hbox{1\!\!\!\!\;\,{I}}}_{[-\varepsilon,\varepsilon]}(x_2(s))ds,\
t\leq \tau_2,
$$
where $(L_{x_2}^0(t))_{t\geq 0}$ is a usual local time of the
process $(x_2(t))_{t\geq 0}$ at zero. There exists a unique strong
solution of equation (\ref{x_2}) (cf.\cite{HS}). Substitution of
solution to (\ref{x_2}) into (\ref{x_1}) allows us to get the
solution of the latter. Hence we have proved existence and
uniqueness of solution to equation (\ref{main}) up to the time
$\tau_2$.

(ii) Put for $m\geq 2$
$$
\tau_m=\inf\left\{t>\tau_{m-1}: x(t)\in  c_{j_m}\ \mbox{for some}\ j_m\in\{1,\dots,n\},j_m\neq j_{m-1}\right\}.
$$
Assume $x(\tau_2)\in c_2$. Then for $\tau_2\leq t<\tau_3$,
the equation (\ref{main}) has the form
\begin{eqnarray}
dx_1(t)&=&dw_1(t)-\gamma_2\frac{\sin(\xi_1+\theta_2)}{\cos\theta_2}dL_{x}^{c_2}(t),\label{x_1tau_2}\\
dx_2(t)&=&dw_2(t)+\gamma_2\frac{\cos(\xi_1+\theta_2)}{\cos\theta_2}dL_{x}^{c_2}(t)\label{x_2tau_2}.
\end{eqnarray}
We do the clockwise rotation of coordinates defined  by the formulas
\begin{eqnarray}
\begin{split}
x'_1&=&x_1\cos\xi_1+x_2\sin\xi_1,\\
x'_2&=&-x_1\sin\xi_1+x_2\cos\xi_1.\label{x'}
\end{split}
\end{eqnarray}
This map takes the ray $c_2$ to the ray $c_1,$ i.e. $c'_2=c_1$. Moreover, it preserves the angles. Using the Ito formula we get the stochastic differential equations
\begin{eqnarray}
dx'_1(t)&=&dw'_1(t)-\gamma_2\tan\theta_2dL_{x'}^{c'_2}(t),\label{x1prime}\\
dx'_2(t)&=&dw'_2(t)+\gamma_2dL_{x'}^{c'_2}(t),\label{x2prime}
\end{eqnarray}
which holds true for $\tau_2\leq t<\tau_3$. Here
$w'_1(t)=\cos\xi_1dw_1(t)+\sin\xi_1dw_2(t),\ t\geq 0,$ and
$w'_2(t)=-\sin\xi_1dw_1(t)+\cos\xi_1dw_2(t),\ t\geq 0,$ are
independent one-dimensional Wiener processes as a square integrable
martingales with characteristics $t$.  By the arguments similar to those of (i)
there exists a unique strong solution to equations (\ref{x1prime}) and (\ref{x2prime}) for $\tau_2\leq t<\tau_3.$ It means, because the inverse transformation exists, that the equations (\ref{x_1tau_2}) and (\ref{x_2tau_2}) has a unique strong solution for $\tau_2\leq t<\tau_3$. The case of $x(\tau_2)\in c_n$ can be
treated analogously. Thus we have built a strong solution to equation (\ref{main}) up to the time $\tau_3$ and proved the uniqueness of it.

(iii) Repeating this procedure of localization of equation
(\ref{main}) we build a solution to (\ref{main}) up to the time
$\zeta=\lim_{m\to\infty}\tau_m$. It is easily seen that the process $(x(t))_{t\geq0}$ does not reach infinity in finite time a.s. and does not intersect any ring $a<|x|<b,\ 0<a<b, x\in \Bbb R^2,$ infinitely many times on $[0,T]$. So,  $\zeta=+\infty$ or $x(\zeta-)=0.$
\end{proof}

\section{Coordinate transformation and change of time}

As above,  we assume that $\xi_k\leq\pi/2, k=1,\dots,n.$

To investigate the value of probability $\Bbb
P_{x^0}\{\zeta<\infty\}$ we make certain transformations of the
process.

Let us change coordinates by the formulas
\begin{eqnarray*}
\tilde x_1 & = &-\ln\sqrt{x_1^2+x_2^2}=-\ln r,\\
\tilde x_2 & = &\varphi,
\end{eqnarray*}
where
$$
r\cos\varphi=x_1,
$$
$$
r\sin\varphi=x_2.
$$
Therefore,
\begin{eqnarray}
x_1&=&e^{-\tilde x_1}\cos\tilde x_2,\label{xtx_1}\\
x_2&=&e^{-\tilde x_1}\sin\tilde x_2.\label{xtx_2}
\end{eqnarray}
The Jacobian matrix  of this transformation is
$$
\frac{\partial(\tilde x_1,\tilde x_2)}{\partial(x_1,x_2)}=
\left(\frac{\partial(x_1,x_2)}{\partial(\tilde x_1,\tilde x_2)}\right)^{-1}=
\begin{bmatrix}
-e^{\tilde x_1}\cos \tilde x_2 & -e^{\tilde x_1}\sin \tilde x_2\\
-e^{\tilde x_1}\sin \tilde x_2 & e^{\tilde x_1}\cos \tilde x_2
\end{bmatrix}.
$$
Moreover, $\Delta\tilde x_1=\Delta\tilde x_2=0$.

The image of the ray $c_k, k=1,\dots,n,$ under this map is the straight-line
$$
\tilde c_k=\{(\tilde x_1,\tilde x_2):\tilde x_1\in\Bbb R, \tilde x_2=\varphi_k\}.
$$
Clearly, the image of $S_k$ is the strip $\tilde S_k=\{(\tilde x_1,\tilde x_2):\tilde x_1\in\Bbb R, \varphi_k\leq\tilde x_2\leq\varphi_{k+1}\}$.
As above we put $\varphi_{nl+k}=\varphi_k, \tilde c_{nl+k}=\tilde c_{k},\ k=1,\dots,n, l\in{\Bbb N}.$

Note that if the function $f$ is twice continuously differentiable on $\Bbb R^2\setminus{\{0\}}$, then by the Ito formula
\begin{equation}
df(x(t))=(\nabla f,dw(t))+\sum_{k=1}^n\gamma_k(\nabla
f,v_k)dL_{x}^{c_k}(t)+\frac{1}{2}\Delta f(x(t))dt, \ t<\zeta.\label{df}
\end{equation}
So
\begin{equation}
\begin{split}d\tilde{x}_1(t)=-e^{\tilde{x}_1(t)}(\cos\tilde{x}_2(t)dw_1(t)+\sin\tilde{x}_2(t)dw_2(t))-\\
-e^{\tilde{x}_1(t)}\sum_{k=1}^n\gamma_k \left(v_{k1}\cos\tilde{x}_2(t)
+v_{k2}\sin\tilde{x}_2(t)\right)dL_{x}^{c_k}(t), \ t<\zeta,\label{tx_1}
\end{split}
\end{equation}
where $(v_{k1},v_{k2})$ are the cartesian coordinates of the vector $v_k$.
As easily seen,
$v_{k1}=-\frac{\sin(\varphi_k+\theta_k)}{\cos\theta_k},v_{k2}=\frac{\cos(\varphi_k+\theta_k)}{\cos\theta_k}.$
The process $(L_{x}^{c_k}(t))_{t\geq 0}$ increases only at those
moments of time when $\tilde x\in \tilde c_k$ and, consequently, $\tilde
x_2=\varphi_k$. Besides, the process $\tilde
w_1(t)=-\cos\tilde{x}_2(t)dw_1(t)-\sin\tilde{x}_2(t)dw_2(t), t\geq
0,$ is a Wiener process as a square integrable martingale with its
characteristic being equal to $t$. Thus, we can rewrite the equation
(\ref{tx_1}) as follows
\begin{equation}
d\tilde{x}_1(t)=e^{\tilde{x}_1(t)}d\tilde
w_1(t)+e^{\tilde{x}_1(t)}\sum_{k=1}^n\gamma_k\tan\theta_k
dL_{x}^{c_k}(t).\label{tx_11}
\end{equation}
By analogy,
\begin{equation}
d\tilde{x}_2(t)=e^{\tilde{x}_1(t)}d\tilde
w_2(t)+e^{\tilde{x}_1(t)}\sum_{k=1}^n\gamma_k
dL_{x}^{c_k}(t),\label{tx_12}
\end{equation}
where
$\tilde w_2(t)=\cos\tilde{x}_2(t)dw_2(t)-\sin\tilde{x}_2(t)dw_1(t),\ t\geq 0,$
is a one-dimensional Wiener process independent of the process
$(\tilde w_1(t))_{t\geq 0}$.

From (\ref{tx_11}), (\ref{tx_12}) we get the equation
\begin{equation}
d\tilde x(t)=e^{\tilde x_1(t)}d\tilde w(t)+e^{\tilde x_1(t)}\sum_{k=1}^n\gamma_k \tilde v_kdL_x^{c_k}(t),\label{tx}
\end{equation}
where $\tilde w(t)=(\tilde w_1(t),\tilde w_2(t)),t\geq 0; \tilde v_k=(\tan\theta_k,1),
k=1,\hdots,n$.

We emphasize that  $L_x^{c_k}$ is not a local time of the process $\tilde x$ but that of the process $x$. Let us give a form of equation (\ref{tx})  that involves $L_{\tilde x}^{\tilde c_k}$ being a local time of the process $\tilde x$ at the line $\tilde c_k$ instead of $L_x^{c_k}, k=1,\dots,n$.

To do this we make use of the following lemma.
\begin{lemma} Let $u(t), \ t\in[0,T],$  be a continuous function on $[0,T]$ such that
for all  $t\in[0,T]$ there exists a limit
$$
\lim_{\varepsilon\downarrow 0}\frac{1}{2\varepsilon}\int_0^t {\hbox{\rm {1}\!\!\!\!\;\,\rm {I}}}_{[-\varepsilon, \varepsilon]}(u(s))ds=:L(t).
$$
Suppose $f$ is a continuous function on $[0,T]$ with positive
values, $g(s,\varepsilon)$ is a function on $[0,T]\times[0,T]$ with
non-negative values such that $g(s,\varepsilon)/\varepsilon\to f(s)$ as
$\varepsilon\downarrow 0$ uniformly in $s$. Then for all $t\in[0.T]$
$$
\lim_{\varepsilon\downarrow 0}\frac{1}{2\varepsilon}\int_0^t {\hbox{\rm {1}\!\!\!\!\;\,\rm {I}}}_{[-g(s,\varepsilon), g(s,\varepsilon)]}(u(s))ds=
\int_0^t f(s)dL(s).
$$\label{Lemma 2}
\end{lemma}
\begin{proof} (i) The statement of Lemma is easily justified for
$g(s,\varepsilon)=\varepsilon{\hbox{1\!\!\!\!\;\,{I}}}_{[a,b]}(s)$, where $a, b$ are
 constant, $0\leq a<b$, and, consequently, it is true for linear combinations of such functions.

(ii) Let $\underset{\!\!\!\!\!\bar{ } }{f_n}, \bar f_n$ be sequences of piecewise
constant functions on $[0,T]$ such that $\underset{\!\!\!\!\!\bar{ } }{f_n}\leq f
\leq\bar f_n$ and $\underset{\!\!\!\!\!\bar{ } }{f_n}\rightarrow f, \bar
f_n\rightarrow f$ as $n\rightarrow\infty$ uniformly. By (i) we
have
$$\int_0^t \underset{\!\!\!\!\!\bar{ } }{f_n}(s)dL(s)=
\lim_{\varepsilon\downarrow 0}\frac{1}{2\varepsilon}\int_0^t
{\hbox{1\!\!\!\!\;\,{I}}}_{[-\varepsilon\underset{\!\!\!\!\!\bar{ } }{f_n}(s),\varepsilon{\underset{\!\!\!\!\!\bar{ } }{f_n}}(s)
]}(u(s))ds\leq
$$
$$
\leq\varliminf_{\varepsilon\downarrow 0}\frac{1}{2\varepsilon}\int_0^t {\hbox{1\!\!\!\!\;\,{I}}}_{[-\varepsilon{f}(s),\varepsilon{f}(s) ]}(u(s))ds
\leq\varlimsup_{\varepsilon\downarrow 0}\frac{1}{2\varepsilon}\int_0^t {\hbox{1\!\!\!\!\;\,{I}}}_{[-\varepsilon{f}(s),\varepsilon{f}(s) ]}(u(s))ds\leq
$$
$$
\leq\lim_{\varepsilon\downarrow 0}\frac{1}{2\varepsilon}\int_0^t
{\hbox{1\!\!\!\!\;\,{I}}}_{[-\varepsilon\bar{f}_n(s),\varepsilon\bar{f}_n(s)
]}(u(s))ds=\int_0^t\bar{f}_n(s)dL(s).
$$
Passing to the limit as $n\rightarrow\infty$ we get
$$
\lim_{\varepsilon\downarrow 0}\frac{1}{2\varepsilon}\int_0^t {\hbox{1\!\!\!\!\;\,{I}}}_{[-\varepsilon{f}(s),\varepsilon{f}(s) ]}(u(s))ds=
\int_0^t f(s)dL(s).
$$

(iii) Since $\frac{g(s,\varepsilon)}{\varepsilon}\to f(s)$ as $\varepsilon\downarrow 0$
uniformly in $s\in [0,T]$, then for each $t\in[0,T],\ \delta>0$ there exists
$\varepsilon_{\delta}>0$ such that for all $\varepsilon<\varepsilon_{\delta}$ the
inequalities
$$
\varepsilon(1-\delta)f(s)\leq g(s,\varepsilon)\leq \varepsilon(1+\delta)f(s)
$$
hold for all $s\in[0,t].$
Similarly to (ii) we have
$$
(1-\delta)\int_0^t f(s)dL(s)\leq \varliminf_{\varepsilon\downarrow
0}\frac{1}{2\varepsilon}\int_0^t {\hbox{1\!\!\!\!\;\,{I}}}_{[- g(s,\varepsilon),
g(s,\varepsilon) ]}(u(s))ds\leq
$$
$$
\leq\varlimsup_{\varepsilon\downarrow 0}\frac{1}{2\varepsilon}\int_0^t
{\hbox{1\!\!\!\!\;\,{I}}}_{[-g(s,\varepsilon), g(s,\varepsilon)]}(u(s))ds\leq
(1+\delta)\int_0^t f(s)dL(s).
$$
Letting $\delta\downarrow 0$ we obtain the statement of
the Lemma.
\end{proof}

Now let us express $L_{\tilde x}^{\tilde c_k}$ in term of $L_x^{c_k}$.

\begin{lemma} The process
$$
L_{\tilde x}^{\tilde c_k}(t)=\int_0^t e^{-\tilde x_1(s)}dL_x^{c_k}(s), t\leq
\zeta,
$$
is a local time of the process $(\tilde x(t))_{t\geq 0}$ on the straight-line $\tilde c_k$.\label{Lemma 3}
\end{lemma}
\begin{proof}  Let $k=1$. According to (\ref{xtx_2})
$$
L_{\tilde x}^{\tilde c_1}(t)=\lim_{\varepsilon\downarrow 0}\frac{1}{2\varepsilon}\int_0^t
{\hbox{1\!\!\!\!\;\,{I}}}_{[-\varepsilon,\varepsilon]}(\tilde x_2(s))ds=
$$
$$
=\lim_{\varepsilon\downarrow
0}\frac{1}{2\varepsilon}\int_0^t{\hbox{1\!\!\!\!\;\,{I}}}_{[-e^{-\tilde x_1(s)}\sin\varepsilon,
e^{-\tilde x_1(s)}\sin\varepsilon]}\hbox{1\!\!\!\!\;\,{I}}_{\{\cos \tilde x_2(s)>0\}}(e^{-\tilde x_1(s)}\sin \tilde x_2(s))ds=
$$
$$
=\lim_{\varepsilon\downarrow
0}\frac{1}{2\varepsilon}\int_0^t{\hbox{1\!\!\!\!\;\,{I}}}_{[-e^{-\tilde x_1(s)}\sin\varepsilon, e^{-\tilde x_1(s)}\sin\varepsilon]}(x_2(s))ds=
$$
$$
=\int_0^t e^{-\tilde x_1(s)}dL_x^{c_1}(s).
$$
The last equality follows from Lemma \ref{Lemma 2}. This proves Lemma \ref{Lemma 3} in the case of $k=1$. The statement of Lemma for $k=2,\dots,n,$ can be obtained analogously.
\end{proof}

Using this Lemma we get the form of equation (\ref{tx}) involving only the process $\tilde x$ as follows
\begin{equation}
d\tilde x(t)=e^{\tilde x_1(t)}d\tilde w(t)+e^{2\tilde x_1(t)}\sum_{k=1}^n\gamma_k
\tilde v_kdL_{\tilde x}^{\tilde c_k}(t).\label{tx_final}
\end{equation}

Let us make a change of time. Put $A(t)=\int_0^t e^{2\tilde x_1(s)}ds,\
t<\zeta,$ and denote $\eta=\int_0^{\zeta} e^{2\tilde x_1(s)}ds$. Define
the process
$$
\hat x(t)=\tilde x(A^{-1}(t)),\ t\leq\eta,
$$
where
$$
A^{-1}(t)=\inf\{s:\int_0^se^{2\tilde x_1(u)}du\geq t\}.
$$

The equation (\ref{tx_final}) implies the stochastic integral
equation
\begin{equation}
\begin{split}\tilde{x}(A^{-1}(t))=\tilde x(0)+\int_0^{A^{-1}(t)}e^{\tilde{x}_1(s)}d\tilde w(t)+\\
+\sum_{k=1}^n\gamma_k\tilde v_k\int_0^{A^{-1}(t)}
e^{2\tilde{x}_1(s)}dL_{\tilde x}^{\tilde c_k}(s), \ t\leq\eta.
\end{split} \label{tx_A{-1}}
\end{equation}
Similarly to \cite{IW}, Theorem 7.2, it can be proved that the
process $\hat w(t)=\int_0^{A^{-1}(t)}e^{\tilde{x}_1(s)}d\tilde w(s)$
is a Wiener process in $\Bbb R^2$ up to the time $\eta$.

\begin{lemma} \label{Lemma 4} The process
$$
L_{\hat x}^{\tilde c_k}(t)=\int_0^{A^{-1}(t)}e^{2\tilde x_1(s)}dL_{\tilde x}^{\tilde c_k}(s), \ t\leq\eta,
$$
is a local time of the process $(\hat x(t))_{t<\eta}$ on the
straight-line $\tilde c_k$.
\end{lemma}

\begin{proof} For $k=1$ we have
$$
L_{\hat x}^{\tilde c_1}(t)=\lim_{\varepsilon\downarrow 0}\frac{1}{2\varepsilon}\int_0^t
{\hbox{1\!\!\!\!\;\,{I}}}_{[-\varepsilon,\varepsilon]}(\hat x_2(s))ds=
$$
$$
=\lim_{\varepsilon\downarrow 0}\frac{1}{2\varepsilon}\int_0^t
{\hbox{1\!\!\!\!\;\,{I}}}_{[-\varepsilon,\varepsilon]}(\tilde x_2(A^{-1}(s)))ds=
$$
$$
=\lim_{\varepsilon\downarrow 0}\frac{1}{2\varepsilon}\int_0^{A^{-1}(t)}
{\hbox{1\!\!\!\!\;\,{I}}}_{[-\varepsilon,\varepsilon]}(\tilde x_2(s))e^{2\tilde x_1(s)}ds=
$$
$$
=\int_0^{A^{-1}(t)}e^{2\tilde x_1(s)}dL_{\tilde x}^{\tilde c_1}(s).
$$
The cases of $k=2,\dots,n$ can be treated analogously.
\end{proof}

Thus using (\ref{tx_A{-1}}) we get the stochastic differential
equation for the process $(\hat x(t))_{t<\eta}$ as follows
\begin{equation}
d\hat x(t)=d\hat w(t)+\sum_{k=1}^n\gamma_k \tilde v_k
dL_{\hat x}^{\tilde c_k}(t)\label{hx}.
\end{equation}

\begin{remark} \label{Remark 2} The  process $(\hat x(t))_{t<\eta}$ is a continuous
Markov one as a result of the random change of time for the Markov
process $(\tilde x(t))_{t\geq0}$ which is continuous (cf. \cite{D}).
So, either $\hat x(t)$ is defined for all $t\geq 0$ or
$\varlimsup_{t\to\eta-}|\hat x(t)|=\infty$. It is easy to see that any solution to SDE (\ref{hx}) does not blow up in finite time a.s. Moreover, in the
next Section we will show that $\Bbb EL_{\hat x_2}^{\tilde c_k}(t)<\infty$
for each $t\geq 0$ (see (\ref{M_k})). Then, by (\ref{hx}), the
trajectories of $\hat x(t)$ do not reach infinity in finite time a.s.
Thus the process $\hat x(t)$ is defined for all $t\geq 0$, i.e. $\Bbb P\{\eta=\infty\}=1$.
\end{remark}

\section {Ergodic behavior of some Markov chain related to the Brownian motion with membranes. }

For $k\in \Bbb N$, let $p_k=\frac{1+\gamma_k}{2},
q_k=\frac{1-\gamma_k}{2}.$ Recall that $\xi_k=\varphi_{k+1}-\varphi_k, \ \tau_m, m\in \Bbb N,$ is defined in the proof of Lemma \ref{Lemma 1}. Put $\xi_0:=\xi_n.$

The process $\hat x_2(t)$
behaves on $(\tau_m,\tau_{m+1}), \ m=1,2,\dots,$ as an ordinary
skew Brownian motion with skewing at the point
$\hat x_2(\tau_m)\in\{\varphi_1,\dots,\varphi_n\}$ (Fig. \ref{Fig. 2}). The skew Brownian motion is a strong Markov
process as a continuous  Feller process (cf. \cite{D}, Theorem 3.10). This yields
$(\hat x_2(\tau_m))_{m\geq1}$ form a homogeneous  Markov chain with the phase space $\{\varphi_1,\dots,\varphi_n\}$ and
the transition matrix
$$\begin{bmatrix}
0 & \tilde p_1 & 0&\hdotsfor{1}&0&\tilde q_1\\
\tilde q_2 & 0&\tilde p_2 & 0&\hdotsfor{1} &0\\
\hdotsfor[3]{6}\\
0&\hdotsfor{1}&0&\tilde q_{n-1}&0&\tilde p_{n-1}\\
\tilde p_n&0&\hdotsfor{1}&0&\tilde q_n&0
\end{bmatrix},
$$
where $$ \tilde
p_k=\Bbb P_{\varphi_{k}}\left\{\hat x_2(\sigma_{\varphi_{k-1}}\wedge\sigma_{\varphi_{k+1}})=\varphi_{k+1}\right\},\tilde q_k=1-\tilde p_k,
$$
$$
\sigma_a=\inf\{t\geq 0: \hat x_2(t)=a\}.
$$

\begin{figure}[h]
  \includegraphics[width=6 cm]{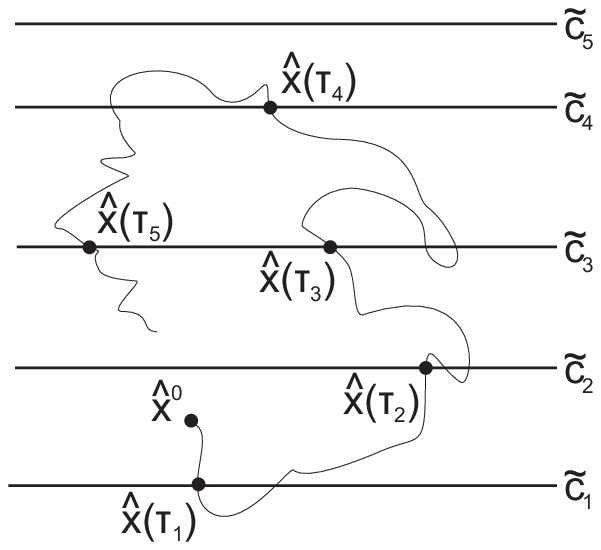}\\
  \caption{}
  \label{Fig. 2}
\end{figure}

One can see (cf. {\cite{P}, Ch. 3}) that, for $k=1,\dots,n$,
\begin{eqnarray}
\tilde p_k&=&\frac{p_k\xi_{k-1}}{p_k\xi_{k-1}+q_k\xi_{k}}\label{tp_k}, \\
\tilde q_k&=&\frac{q_k\xi_{k}}{p_k\xi_{k-1}+q_k\xi_{k}},\label{tq_k}.
\end{eqnarray}

Now we find $
M_k:=\Bbb{E}_{\varphi_k}L_{\hat x_2}^{\varphi_k}(\sigma_{\varphi_{k-1}}\wedge\sigma_{\varphi_{k+1}}).$
Equation (\ref{hx}) implies
$$
\Bbb
E_{\varphi_k}\hat x_2(\sigma_{\varphi_{k-1}}\wedge\sigma_{\varphi_{k+1}})=\varphi_k+\gamma_k\Bbb{E}_{\varphi_k}L_{\hat x_2}^{\tilde c_k}(\sigma_{\varphi_{k-1}}\wedge\sigma_{\varphi_{k+1}}).
$$
Using (\ref{tp_k}),(\ref{tq_k}) we get
\begin{equation}
M_k=\frac{\xi_{k-1}\xi_k}{p_k\xi_{k-1}+q_k\xi_k}.\label{M_k}
\end{equation}
{\bf 3.1.} Assume $0<p_k<1, \ k=1,\dots,n$. The Markov chain $(\hat x_2(\tau_m))_{m\geq 1}$ is irreducible and
all of its states are positive-recurrent. This implies the
existence of a unique invariant distribution $(\pi_1,\dots,\pi_n)$
which is the solution to the following system of equations
\begin{eqnarray}
\pi_k&=&\pi_{k-1}\tilde p_{k-1}+\pi_{k+1}\tilde q_{k+1}, \ k=2,\dots,n-1,\label{pi_k}\\
\pi_1&=&\pi_n\tilde p_n+\pi_2\tilde q_2,\label{pi_1}\\
\sum_{k=1}^n\pi_k&=&1.\label{sum}
\end{eqnarray}
Let us find it out.  We can rewrite the equation (\ref{pi_k}) in
the form
$$
\pi_k(\tilde p_k+\tilde q_k)=\pi_{k-1}\tilde p_{k-1}+\pi_{k+1}\tilde q_{k+1},
k=2,\dots,n-1.
$$
This implies the equality
$$
\pi_{k+1}\tilde q_{k+1}-\pi_k\tilde p_k=\pi_k\tilde q_k-\pi_{k-1}\tilde p_{k-1}
$$
valid for $k=2,\dots,n-1$. Put
\begin{equation}
C:=\pi_{k+1}\tilde q_{k+1}-\pi_k\tilde p_k, \ k=1,\dots,n-1.\label{C}
\end{equation}
Let $k=1$. Then from (\ref{C}) we have
\begin{equation}
\pi_2=C\frac{1}{\tilde q_2}+\pi_1\frac{\tilde p_1}{\tilde q_2}. \label{pi_2}
\end{equation}
Put $k=2$ in (\ref{C}). Taking into account (\ref{pi_2}) we obtain
$$
\pi_3=C\left(\frac{1}{\tilde q_3}+\frac{\tilde p_2}{\tilde q_2\tilde q_3}\right)+\pi_1\frac{\tilde p_1\tilde p_2}{\tilde q_2\tilde q_3}.
$$

Finally by recursion we arrive at the formula
$$
\pi_k=C\left( \frac{1}{\tilde q_k}+\frac{\tilde p_{k-1}}{\tilde q_{k-1}\tilde q_k}+\cdots+\frac{\tilde p_2\cdot\hdots\cdot \tilde p_{k-1}}{\tilde q_2\cdot\hdots\cdot\tilde q_k}\right)+\pi_1\frac{\tilde p_1\cdot\hdots\cdot \tilde p_{k-1}}{\tilde q_2\cdot\hdots\cdot\tilde q_k}
$$
justified for $k=2,\dots,n$.

Substituting this expression for $\pi_n$ in the equality
$$
\pi_1\tilde q_1-\pi_n\tilde p_n=C,
$$
which is a simple consequence of (\ref{pi_1}), we get
$$
\pi_1=\frac{1+\tilde r_n+\tilde r_{n-1}\tilde r_n+\cdots+\tilde r_2\cdot\hdots\cdot
\tilde r_n}{\tilde q_1(1-\tilde r_1\cdot\hdots\cdot \tilde r_n)}\ C , \hbox{ if }
\tilde p_1\cdot\hdots\cdot\tilde p_n\neq\tilde q_1\cdot\hdots\cdot\tilde q_n,
$$
where  $\tilde r_k=\frac{\tilde p_k}{\tilde q_k}, k=1,2,\dots,n$.
Note that $\tilde p_1\cdot\hdots\cdot\tilde p_n\neq\tilde q_1\cdot\hdots\cdot\tilde q_n$ if and only if $p_1\cdot\hdots\cdot p_n\neq q_1\cdot\hdots\cdot q_n$.
Taking into account that we can chose the ray $\tilde c_1$ in an arbitrary way, we get for $k=1,2,\dots,n  $
\begin{equation*}
\pi_k=\frac{1+\tilde r_{k-1}+\tilde r_{k-2}\tilde r_{k-1}+\cdots+\tilde r_1\cdot\hdots\cdot \tilde r_{k-1}\tilde r_{k+1}\cdot\hdots\cdot \tilde r_n}{\tilde q_k(1-\tilde r_1\cdot\hdots\cdot \tilde r_n)}\ C,
\end{equation*}
if $p_1\cdot\hdots\cdot p_n\neq q_1\cdot\hdots\cdot q_n,$

One can find the value of $C$ by substitution of  the last formula into (\ref{sum}).
Computations lead to the formula
\begin{equation}
\pi_k=C_1\beta_k,\label{pi_kC}
\end{equation}
where
$$
\beta_k=\frac{\sum_{j=0}^{n-1}q_{k+1}\cdot\hdots\cdot q_{k+j}p_{k+j+1}\cdot\hdots\cdot p_{n+k-1}\xi_{k+1}\cdot\hdots\cdot \xi_{k+j}\xi_{k+j}\cdot\hdots\cdot \xi_{n+k-2}}{\prod_{i=1, \ i\neq k}^n(p_i\xi_{i-1}+q_i\xi_i)},
$$
$$
C_1=1/\sum_{k=1}^n\beta_k.
$$
If $p_1\cdot\hdots\cdot p_n=q_1\cdot\hdots\cdot q_n,$ it is
easy to see that
\begin{eqnarray}
\begin{split}
\pi_1&=&\frac{1}{1+\frac{\tilde p_1}{\tilde q_2}+\frac{\tilde p_1\tilde p_2}{\tilde q_2\tilde q_3}+\cdots+\frac{\tilde p_1\cdot\hdots\cdot\tilde p_{n-1}}{\tilde q_2\cdot\hdots\cdot \tilde q_{n}}}&,\\
\pi_k&=&\pi_1\frac{\tilde p_1\cdot\hdots\cdot\tilde p_{k-1}}{\tilde q_2\cdot\hdots\cdot\tilde q_k},&\ k=2,\dots,n.\label{pi_kC0}
\end{split}
\end{eqnarray}

Now suppose $0<p_k\leq1, \ k=1,\dots,n,$ or $0\leq p_k<1, \ k=1,\dots,n.$  It is easy to check that formula (\ref{pi_kC}) has no singularities in this case.

Thus we have proved the following statement

\begin{lemma}\label{Lemma 5} Let, for $k=1,\dots,n, \ 0\leq p_k<1$ or let, for $k=1,\dots,n, \ 0< p_k\leq1$. Then there
exists a unique invariant distribution  $(\pi_k)_{k=1}^n$ of the
Markov chain $(\hat x_2(\tau_m))_{m\geq1}.$

If $p_1\cdot\hdots\cdot p_n\neq q_1\cdot\hdots\cdot q_n$ this invariant distribution is given by formula (\ref{pi_kC}), $k=1,\dots,n$.

If $p_1\cdot\hdots\cdot p_n= q_1\cdot\hdots\cdot q_n$ the invariant distribution is given by formula (\ref{pi_kC0}), $k=1,\dots,n$.
\end{lemma}
\begin{remark} The method of finding the invariant distribution was proposed by Professor Alexey Kulik.
\end{remark}

Now partial cases when some of $p_k$ are equal to $0$ and some of them are equal to $1$ are left unattended. We consider them in Subsections 3.3, 3.4.

{\bf 3.2.} Further on we make use of the following well-known result. Given a
homogeneous Markov chain with state space \{0,1,\dots, N\} and
transition matrix
$$\begin{bmatrix}
0 & 1 & 0&\hdotsfor{2}&0&\\
x_1 & 0&y_1 &0&\hdots&0\\
\hdotsfor[3]{6}\\
0&\hdots&0&x_{N-1}&0&y_{N-1}\\
0&\hdotsfor{2}&0&1&0
\end{bmatrix},
$$
where $0<x_i<1,0<y_i<1, \ i=1,\dots,N-1.$

As is known there exists a unique invariant distribution
$$\{f_i\}_{i=0}^N=\{f_ш(x_1,\dots,x_{N-1},y_1,\dots,y_{N-1},N)\}_{i=0}^N$$ where
\begin{eqnarray}
f_0&=&\frac{1}{1+\frac{1}{x_1}+\frac{y_1}{x_1x_2}+\dots+\frac{y_1\cdot\hdots\cdot
y_{N-2}}{x_1\cdot\hdots\cdot x_{N-1}}\label{f_0}
+\frac{y_1\cdot\hdots\cdot y_{N-1}}{x_1\cdot\hdots\cdot x_{N-1}}},\\
f_1&=&\frac{1}{x_1}f_0,\label{f_1}\\
f_i&=&\frac{y_1\cdot\hdots\cdot y_{i-1}}{x_1\cdot\hdots\cdot
x_{i}}f_0, \ i=2,\dots,N-1,\label{f_i}\\
f_N&=&y_{N-1}f_{N-1}.\label{f_N}
\end{eqnarray}

{\bf 3.3.} Suppose $p_{k_0}=1,\ p_{k_1}=\dots=p_{k_l}=0,\ ,
k_0,k_1,\dots, k_l\in\{1,\dots,n\}$, and $0<p_k < 1, \
k\in\{1,\dots,n\}\setminus \{k_0,k_1,\dots,k_l\}.$ Without loss
of generality we may assume that $k_0<k_1<\dots<k_l$.  Then the
Markov chain $(\hat x_2(\tau_m))_{m\geq1}$ has a unique communicating
positive-recurrent class
$\{\varphi_{k_0},\varphi_{k_0+1},\dots,\varphi_{k_1}\}.$ All other states are
transient. So there exists a unique invariant distribution
$(\pi_k)_{k=1}^n$. Besides, it is not difficult to verify that
\begin{eqnarray*}
\pi_k&=&0, \ 1\leq k<k_0 \ \hbox{ or } \ k_1<k\leq n,\\
\pi_{k_0+i}&=&f_i(\tilde q_{k_0+1},\dots,\tilde q_{k_1-1},\tilde p_{k_0+1},\dots,\tilde p_{k_1-1},k_1-k_0),\\&& i=0,1,\dots,k_1-k_0.
\end{eqnarray*}
where $f_i, i=0,1,\dots,k_1-k_0,$ is defined by formulas (\ref{f_0}) -- (\ref{f_N}).

The case of $p_{k_0}=0$ and $0<p_k\leq 1, \ k\in\{1,\dots,n\}\setminus\{k_0\}$ can be treated analogously.

{\bf 3.4.}  Suppose there exist $k_1,k_2,\dots,k_{2l}\in\{1,\dots,n\}, k_1<k_2<\dots<k_{2l},$ such that $p_{k_1}=p_{k_3}=\dots=p_{k_{2l-1}}=1,\ p_{k_2}=p_{k_4}=\dots=p_{k_{2l}}=0$, and $0<p_k<1,\ k_{2j-1}<k<k_{2j},\ j=1,\dots,l.$ Then each of the classes $\{\varphi_{k_1},\varphi_{k_1+1},\dots,\varphi_{k_2}\},\dots,$ $\{\varphi_{k_{2l-1}},\varphi_{k_{2l-1}+1},\dots,\varphi_{k_{2l}}\}$ is a communicating one for $(\hat x_2(\tau_m))_{m\geq 1}.$ (Elements of different classes are not communicate.) For each class, there exists a unique invariant distribution which can be found by formulas (\ref{f_0})--(\ref{f_N}).

Assume $\hat x_2(\tau_1)=\varphi_{k_0},\ k_{2j}<k_0<k_{2j+1}$ for some $j\in\{1,\dots,l\}.$ $(k_{2l+1}:=k_1).  $ Define
$$
\alpha(k_0)=\Bbb P_{\varphi_{k_0}}\{\hat x_2(\sigma_{\varphi_{2j}}\wedge\sigma_{\varphi_{2j+1}})
=\varphi_{2j}\}.
$$
Note that
$$
\Bbb P_{\varphi_{k_0}}\{\sigma_{\varphi_{2j}}\wedge\sigma_{\varphi_{2j+1}}<\infty)=1.
$$
If $0<p_k<1,\ k_{2j}<k<2j+1$,
computations lead to the formula
$$
\alpha(k_0)=
\frac{\sum_{i=k_0}^{k_{2j+1}}\frac{\tilde q_{k_{2j}+1}\cdot\hdots\cdot\tilde q_i}{\tilde p_{k_{2j}+1}\cdot\hdots\cdot\tilde p_i}}
{{1+\sum_{i=k_{2j}+1}^{k_{2j+1}}\frac{\tilde q_{k_{2j}+1}\cdot\hdots\cdot\tilde q_i}{\tilde p_{k_{2j}+1}\cdot\hdots\cdot\tilde p_i}}}.
$$

If $0\leq p_k<1, k_{2j}<k<k_{2j+1},$ and equality is reached, $\alpha(k_0)=1$. Similarly, if $0<p_k\leq1, k_{2j}<k<k_{2j+1},$ and equality is reached, $\alpha(k_0)=0$.

If there exist $i_1,\dots,i_s,i_{s+1},\dots,i_r$ such that $k_{2j}<i_1<\dots<i_s<k_0<i_{s+1}<\dots<i_r<k_{2j+1},$ $p_{i_1}=\dots=p_{i_s}=1,$ $p_{i_{s+1}}=\dots=p_{r}=0.$ Then
$$
\alpha(k_0)=\Bbb P_{\varphi_{k_0}}\{\hat x_2(\sigma_{\varphi_{s}}\wedge\sigma_{\varphi_{s+1}})
=\varphi_{s}\}.
 $$
\vskip 5 pt

\section {The main result. }

The behavior of the process $(x(t))_{t\geq 0}$ is uniquely
determined by that of the process $(\hat x(t))_{t\geq 0}$.  As we
know (see Remark \ref{Remark 2})  $\eta=\infty $ a.s. So $x(t)\to 0$ as
$t\to\zeta-$ if and only if $\hat x(t)\to+\infty$ as $t\to+\infty$.

Let us obtain an estimate of the local time $L_{\hat x}^{\tilde c_k}(t), \
k=1,\dots,n$. Fix $k\in\{1,\dots,n\}.$ Then the  functional
$L_{\hat x}^{\tilde c_k}(t)$ increases at those and only those intervals
of time $[\tau_m,\tau_{m+1}), m\geq 1,$ for which $\hat x(\tau_m)\in
\tilde c_k$. Moreover,
$$
L_{\hat x}^{\tilde c_k}(t)=L_{\hat x_2}^{\varphi_k}(t)=: \lim_{\varepsilon\downarrow 0}\frac{1}{2\varepsilon}\int_0^t{\hbox{1\!\!\!\!\;\,{I}}}_{[\varphi_k-\varepsilon,\varphi_k+\varepsilon]}(\hat x_2(s))ds.
$$
Let $\nu_k(t)$ be the number of
such intervals on $[0,t]$. According to the strong Law of large numbers
$$
\frac{L_{\hat x}^{\tilde c_k}(t)}{\nu_k(t)}\to \Bbb{E}_{\varphi_k}L_{\hat x_2}^{\varphi_k}(\sigma_{\varphi_{k-1}}\wedge\sigma_{\varphi_{k+1}}) \hbox{ as } t\to+\infty \hbox{ a.s.}
$$

Fix $\varepsilon>0$. Then there exists $t_{k1}(\omega)>0$ such that for all $t>t_{k1}$
the inequality
\begin{equation}
\nu_k(t)M_k(1-\varepsilon)\leq
L_{\hat x}^{\tilde c_k}(t)\leq\nu_k(t)M_k(1+\varepsilon)\label{Lve}
\end{equation}
holds. Recall that
$M_k:=\Bbb{E}_{\varphi_k}L_{\hat x_2}^{\varphi_k}(\sigma_{\varphi_{k-1}}\wedge\sigma_{\varphi_{k+1}})$
and  it is defined by formula (\ref{M_k}).

Further, by the renewal theorem (cf. \cite{F}, Ch. XI),
$$
\frac{\Bbb E\nu_k(t)}{t}\to K\pi_k \hbox{ as } t\to\infty,
$$
where $K$ is some positive constant. Using the strong law of large
numbers we get
$$
\frac{\nu_k(t)}{t}\to K\pi_k \hbox { as } t\to \infty \hbox{ a.s.}
$$
Thus there exists $t_{k2}\geq t_{k1}$ such that for all $t>t_{k2}$
\begin{equation}
K\pi_k t(1-\varepsilon)\leq\nu_k(t)\leq K\pi_k t(1+\varepsilon) \hbox{
a.s.}.\label{nu_t}
\end{equation}
Finally from (\ref{Lve}) and (\ref{nu_t}) we obtain the inequality
\begin{equation}
\begin{split}
(1-\varepsilon)^2K\pi_k t M_k&\leq
L_{\hat x}^{\tilde c_k}(t)\leq\\
&\leq(1+\varepsilon)^2K\pi_k t M_k, \ t>t_{k2}, \hbox { a.s.}\label{ineq}
\end{split}
\end{equation}
We use the inequality  (\ref{ineq}) to get an estimation of the
process $(\hat x_1(t))_{t\geq 0}$.

 Let $\sum_{k=1}^n
\gamma_k\pi_k M_k\tan\theta_k\neq 0$. Integrating (\ref{hx}) over
$[0,t]$ and applying  (\ref{ineq}) we come to the inequality
\begin{equation}
\begin{split}
\hat x_1(0)+\hat w_1(t)+Kt\sum_{k=1}^n\gamma_k(1-\varepsilon_k)^2\pi_k M_k\tan\theta_k\leq\\
\leq\hat x_1(t)\leq\\
\leq\hat x_1(0)+\hat w_1(t)+Kt\sum_{k=1}^n\gamma_k(1+\varepsilon_k)^2\pi_k M_k\tan\theta_k\label{ineq_to}
\end{split}
\end{equation}
valid for all $t>t^*=\max_{k=\overline{1,n}}\{t_{12},\dots,t_{n2}\},$
where
$$
\varepsilon_k=
\begin{cases}
\varepsilon&\text{ if $\gamma_k\tan\theta_k>0,$}\\
-\varepsilon&\text{ if $\gamma_k\tan\theta_k<0.$}
\end{cases}
$$
So if $\sum_{k=1}^n \gamma_k\pi_k M_k\tan\theta_k\neq 0$ we can
choose $\varepsilon$ so small that the expressions
$$
\sum_{k=1}^n \gamma_k\pi_k M_k\tan\theta_k,
$$
$$
\sum_{k=1}^n \gamma_k(1+\varepsilon_k)^2\pi_kM_k\tan\theta_k,
$$ and
$$
\sum_{k=1}^n \gamma_k(1-\varepsilon_k)^2\pi_kM_k\tan\theta_k
$$
 have the same sign. By the low of iterated logarithm  there exists $t_*>0$ such that for all $t>t_*$ the inequality
\begin{equation}
|\hat w_1(t)|\leq (1+\varepsilon)\sqrt{2t\ln\ln t}\label{ILL}
\end{equation}
fulfilled. Combining (\ref{ineq_to}) and (\ref{ILL}) we see that
\begin{eqnarray*}
\hat x_1(t)\to +\infty &\mbox{as } \ t\to +\infty &\hbox{if } \ \sum_{k=1}^n \gamma_k\pi_k M_k\tan\theta_k>0 \\
\hat x_1(t)\to -\infty & \hbox{as } \ t\to +\infty &\hbox{if } \
\sum_{k=1}^n \gamma_k\pi_k M_k\tan\theta_k<0.
\end{eqnarray*}

Now we can return from the process $(\hat x(t))_{t\geq 0}$ to
$(x(t))_{t\geq 0}$ making the inverse change of time and doing the
rotation of coordinate system inverse to that defined by
(\ref{x'}).  Then we see that $x(t)\to 0$, as $t\to\zeta-$, if
$\sum_{k=1}^n (p_k-q_k)\pi_k
\frac{\xi_{k-1}\xi_k}{p_k\xi_{k-1}+q_k\xi_k}\tan\theta_k>0$. In this case, as
a consequence of (\ref{ineq_to}), there exist $t_0, K_1, K_2>0$
such that for all $t>t_0$ the inequality $\hat x_1(t)>K_1+K_2t$ holds
true a.s. This implies convergence of the integral $\int_0^\infty
e^{-2\hat x_1(s)}ds$. Therefore, $\zeta<\infty$ a.s. Similarly, if
$\sum_{k=1}^n (p_k-q_k)\pi_k
\frac{\xi_{k-1}\xi_k}{p_k\xi_{k-1}+q_k\xi_k}\tan\theta_k<0$, then
$\zeta=+\infty$ a.s. and $x(t)\to\infty$ as $t\to\infty$ a.s.

Let $\sum_{k=1}^n\gamma_k\pi_k M_k\tan\theta_k=0$. Define
$$
\mu_1=\inf\{t\geq 0: \hat x(t)\in\tilde c_1\},\\
$$
and, for $j\geq2$,
$$
\mu_j=\inf\{t>\tilde \mu_{j-1}: \hat x(t)\in\tilde c_1\},\\
$$
where
$$
\tilde \mu_{j}=\inf\left\{t>\mu_j: \hat x(t)\in\{\tilde c_n, \tilde c_2\}\right\}
$$
(see Fig. \ref{Fig. 3}).
\begin{figure}[h]
  \includegraphics[width=6 cm]{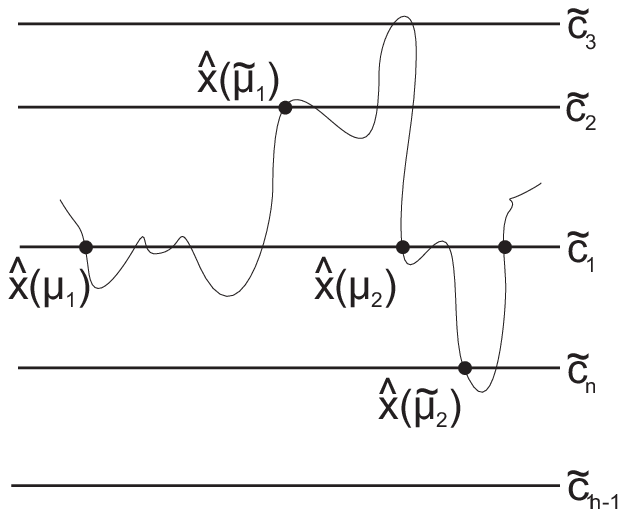}
  \caption{}
  \label{Fig. 3}
\end{figure}

Denote
$$
\psi_j=\sum_{k=1}^n\gamma_k\tan{\theta_k}(L_{\hat x}^{\tilde c_k}(\mu_{j+1})-L_{\hat x}^{\tilde c_k}(\mu_j)),\ j\geq 1.
$$
It is easily seen that if $\sum_{k=1}^n\gamma_k\pi_k M_k\tan\theta_k=0$ then $\Bbb E \psi_j=0,\ j\geq 1$.
From (\ref{hx}) we get
$$
\hat x_1(\mu_j)=\hat x_1(\mu_1)+\sum_{l=1}^j\left(\hat w_1(\mu_{l+1})-\hat w_1(\mu_l)+\psi_l\right).
$$
Note that $\chi_l=\hat w_1(\mu_{l+1})-\hat w_1(\mu_l)+\psi_l,\  l\geq 1,$ are independent identically distributed random variables (i.i.d.r.v.), $\Bbb E\chi_l=0,$ and all the moments are finite. Then by law of iterated logarithm for sum of i.i.d.r.v. (cf. \cite{Pet}) we have
$$
\varlimsup_{j\to\infty}\hat x_1(\mu_j)=+\infty, \ \varliminf_{j\to\infty}\hat x_1(\mu_j)=-\infty\ \ \hbox {a.s.}
$$
It is easily seen that in this case the integral $\int_0^\infty e^{-2\hat x_1(s)}ds$ diverges a.s. Therefore $\zeta=+\infty$ a.s.

Recall that conventionally $c_{nl+k}=c_k,\ \xi_{nl+k}=\xi_{k}$ and put additionally $\gamma_{nl+k}=\gamma_k,\pi_{nl+k}=\pi_k,\ k=1,\dots,n, l\in{\Bbb N}.$

\begin{definition} We call $c_i0c_j,1\leq i<j,$ a characteristic wedge if $\varphi_i,\varphi_{i+1},\dots,\varphi_j$ is a communicating class for the Markov chain $(\hat x_2(\tau_m))_{m\geq1}$.
\end{definition}

\begin{theorem}\label{Theorem} {\it Let $x(0)=x^0, \ x^0\neq0.$ Then

1) $\Bbb P_{x_0}\{\zeta<\infty\}+\Bbb P_{x_0}\{\varlimsup_{t\to+\infty}|x(t)|=\infty\}=1$.

2) The probability $\Bbb P_{x_0}\{\zeta<\infty\}$ is equal to the probability of hitting some characteristic wedge $c_{j_1}0c_{j_2},1\leq j_1<j_2,$ for which $$\sum_{k=j_1}^{k=j_2}(p_k-q_k)\pi_k
\frac{\xi_{k-1}\xi_k}{p_k\xi_{k-1}+q_k\xi_k}\tan\theta_k>0$$ where $\pi_k,\ j_1\leq k\leq j_2,$ is defined  in Section 3.

In particular, suppose $0\leq p_k<1$ or $0<p_k\leq1, \ k=1,\dots,n.$ Then, if $\sum_{k=1}^n (p_k-q_k)\pi_k
\frac{\xi_{k-1}\xi_k}{p_k\xi_{k-1}+q_k\xi_k}\tan\theta_k>0$, the process
$(x(t))_{t\geq 0}$ hits the origin a.s.;
if $\sum_{k=1}^n (p_k-q_k)\pi_k
\frac{\xi_{k-1}\xi_k}{p_k\xi_{k-1}+q_k\xi_k}\tan\theta_k<0$, the process
$(x(t))_{t\geq 0}$ does not hit the origin a.s. and
$|x(t)|\to\infty$ as $t\to+\infty$ a.s.; if $\sum_{k=1}^n (p_k-q_k)\pi_k
\frac{\xi_{k-1}\xi_k}{p_k\xi_{k-1}+q_k\xi_k}\tan\theta_k=0$ the process does not hit the origin  a.s. and $\varlimsup_{t\to+\infty}|x(t)|=\infty$ a.s.}
\end{theorem}

\begin{remark}\label{Remark 3} It is not difficult to calculate that if $0\leq p_k<1, p_1\cdot\hdots\cdot p_n\neq q_1\cdot\hdots\cdot q_n,$ the process $(x(t))_{t\geq 0}$ hits the origin if and only if
$$
\sum_{k=1}^n\xi_k\sum_{j=1}^n r_{k+1}\cdot\hdots\cdot r_{k+j-1}(r_{k+j}-1)\tan\theta_{k+j}>0
$$
where $r_j=p_j/q_j,\ j=1,\dots,n.$
\end{remark}

\section{Examples}

{\bf Example 1.} Let $n=1$. This is a degenerate case in which the process reaches the origin if and only if $\tan\theta_1>0,$ i.e. $\theta_1>0$.
\newline
{\bf Example 2.} Let $n=2, \xi_1=\xi, 0\leq p_1,p_2\leq 1$. Then $\pi_1=\pi_2=1/2$ and the
process hits the origin if and only if
\begin{multline*}
(p_1-q_1)(p_2-q_2)\xi(\tan\theta_1-\tan\theta_2)+\\+2\pi\left((p_1-q_1)q_2\tan\theta_1+(p_2-q_2)p_1\tan\theta_2\right)>0.
\end{multline*}
Consider the particular case when $p_1=1,\ p_2=0.$
Then the part of the process in the interior of the wedge
$c_10c_2$ is a Brownian motion with instantaneous reflection at
the boundary of the wedge (Fig. \ref{Fig. 4}).
\begin{figure}[h]
  \includegraphics[width=6 cm]{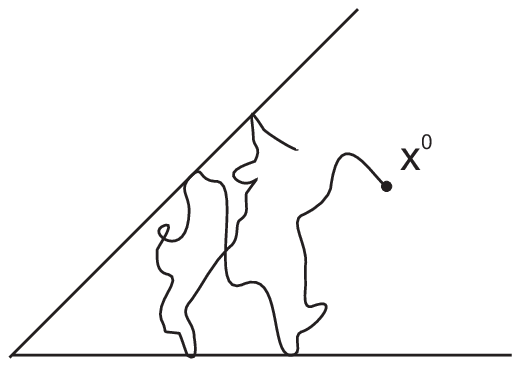}
  \caption{}
  \label{Fig. 4}
\end{figure}
Let the starting point $x^0$
be inside the wedge and away from the origin.
In this case the process hits the origin if and only if
$\tan\theta_1-\tan\theta_2>0$, i.e. $\theta_1-\theta_2>0$. Thus
the result we obtained coincides with that of \cite{VW}.
\newline
{\bf Example 3.} Let $n=3, 0<p_k\leq1$ or $0\leq p_k\leq 1, k=1,2,3.$ If $p_1p_2p_3\neq q_1q_2q_3$, the invariant distribution has the form
\begin{eqnarray*}
\pi_1&=&\frac{\tilde q_2+\tilde p_2\tilde p_3}{D},\\
\pi_2&=&\frac{\tilde q_3+\tilde p_1\tilde p_3}{D},\\
\pi_3&=&\frac{\tilde q_1+\tilde p_1\tilde p_2}{D},
\end{eqnarray*}
where $D=\tilde q_1+\tilde q_2+\tilde q_3+\tilde p_1\tilde p_2+\tilde p_2\tilde p_3+\tilde p_1\tilde p_3$. In this case the process hits the origin if and only if
\begin{multline*}
(p_1-q_1)\tan\theta_1(q_2p_3\xi_2+q_2q_3\xi_3+p_2p_3\xi_1)+\\
+(p_2-q_2)\tan\theta_2(p_1q_3\xi_3+q_1q_3\xi_1+p_1p_3\xi_2)+\\
+(p_3-q_3)\tan\theta_3(q_1p_2\xi_1+q_1q_2\xi_2+p_1p_2\xi_3)>0.
\end{multline*}

If $p_1p_2p_3= q_1q_2q_3$ we get
\begin{eqnarray*}
\pi_1&=&\frac{\tilde q_2\tilde q_3}{\tilde q_2\tilde q_3+\tilde p_1\tilde q_3+\tilde p_1\tilde p_2},\\
\pi_2&=&\frac{\tilde p_1\tilde q_3}{\tilde q_2\tilde q_3+\tilde p_1\tilde q_3+\tilde p_1\tilde p_2},\\
\pi_3&=&\frac{\tilde p_1\tilde p_2}{\tilde q_2\tilde q_3+\tilde p_1\tilde q_3+\tilde p_1\tilde p_2}
\end{eqnarray*}
and the process hits the origin if and only if
$$
q_2q_3(p_1-q_1)\tan\theta_1+p_1q_3(p_2-q_2)\tan\theta_2+p_1p_2(p_3-q_3)\tan\theta_3>0.
$$

\vskip 15 pt

\makeatother

Institute of Geophysics, National Academy of Sciences of Ukraine,
Palladina pr. 32, 03680, Kiev-142, Ukraine,\newline
oaryasova$@$mail.ru
\vskip 10 pt

Institute of Mathematics,  National Academy of Sciences of
Ukraine, Tereshchenkivska str. 3, 01601, Kiev, Ukraine\newline
apilip$@$imath.kiev.ua
\end{document}